\documentclass[graybox]{svmult}

\usepackage{type1cm}        
\usepackage{makeidx}         
\usepackage{graphicx}        
\usepackage{multicol}        
\usepackage[bottom]{footmisc}

\usepackage{newtxtext}       %
\usepackage[varvw]{newtxmath} 

\usepackage{tikz}%
\usepackage{amsmath}%
\usepackage{amsfonts}%

\newcommand{\RR}{{\mathbb R}}

\newcommand{\pp}{{\mathbb P}}

\providecommand{\U}[1]{\protect\rule{.1in}{.1in}}

\definecolor{drot}{rgb}{0.01, 0.28, 1.0}

\makeindex             

\begin{document}

\title*{Construction of 2D explicit cubic quasi-interpolating splines in
Bernstein-B\'{e}zier form}
\titlerunning{2D quasi-interpolation in Bernstein-B\'{e}zier form}
\author{D. Barrera, S. Eddargani, M.J. Ib\'{a}\~{n}ez and S. Remogna}
\institute{Domingo Barrera and Mar\'{\i}a Jos\'{e} Ib\'{a}\~{n}ez \at Department of Applied Mathematics, University of Granada, Campus de Fuentenueva s/n, 18071-Granada, Spain, \email{\{dbarrera,mibanez\}@ugr.es}
\and Salah Eddargani  \at Department of Mathematics, University of Rome Tor Vergata, Rome, Italy, \email{eddargani@mat.uniroma2.it} \and Sara Remogna \at University of Torino,  Via Carlo Alberto 10, 10123, Torino, Italy, \email{sara.remogna@unito.it} }
\maketitle

\abstract{In this paper, the construction of $C^{1}$ cubic quasi-interpolants 
on a three-direction mesh of $\RR^{2}$ is addressed. The 
quasi-interpolating splines are defined by directly setting their 
Bernstein-B\'{e}zier coefficients relative to each triangle from point and 
gradient values in order to reproduce the polynomials of the highest possible 
degree. Moreover, additional global properties are required. Finally, we provide some numerical tests confirming the approximation properties.}

\section{Introduction} \label{intro}

In many scientific applications and mathematical problems the approximation of 
functions from their values or some derivatives at given points is present, 
and quasi-interpolation is a simple and useful procedure in this context 
thanks to its particular properties (see e.g. the book \cite{bj} for a general 
overview on this topic). Indeed, the construction of classical approximants, 
e.g. interpolants, often requires the resolution of linear systems, instead 
quasi-interpolants are local approximants avoiding this problem. 

Here we focus on spline quasi-interpolation and we recall there are several 
schemes that allow to represent them (see e.g. the book \cite{ls} and the 
reference therein), for example using compactly supported spanning functions, 
like B-splines or box splines, or using local and stable minimal determining 
sets. Starting from \cite{SZ3,SorokinaZeilfelder2005} and going on with \cite{bcir2,bcir3,bcir4,beir,Proc2019,MMAS, BEIR23}%
, another local approach has been adopted in the literature and it is based on 
the Bernstein-B\'{e}zier (BB-) representation of polynomials, by setting the 
spline BB-coefficients to appropriate combinations of the given data values, 
by using local portions of the data in such a way that the $C^{1}$ smoothness 
conditions are satisfied as well as the required polynomial reproduction. 
In particular, in \cite{bcir2,bcir3,bcir4} $C^1$ quartic and cubic quasi-interpolants on type-1 triangulations, exact on the space of cubic and quadratic polynomials, respectively, are constructed. In \cite{beir} such a method has been applied for the construction of $C^1$ quadratic quasi-interpolants exact on quadratic polynomials, defined on a uniform triangulation of type-1 endowed with a Powell–Sabin refinement.  In \cite{Proc2019} the method has been modified by combining a quasi-interpolating spline with one step of the so called Modified Buttertly Interpolatory Subdivision Scheme, to construct $C^1$ quartic interpolating splines on regular type-1 triangulations. Moreover, in \cite{MMAS}, quasi-interpolating schemes constructed by using this method have been applied to digital elevation models.

We remark that in the above papers, the BB-coefficients are determined using only the values of the function to be approximated. In this context, in the present paper we propose the construction of $C^1$-cubic Hermite splines on a uniform three-direction triangulation, whose BB-coefficients are determined by the values of the function and its gradient at the vertices of the triangulation and the associated quasi-interpolation operator is exact on quadratic polynomials. The resulting spline (obtained by imposing $C^1$ smoothness and quadratic polynomial reproduction) depends on five parameters that we fix imposing additional properties.

In particular, in Section \ref{Notations_preliminaries} we give notations and preliminaries used in the paper. In Section \ref{cubic} we define the problem and we prove the existence of a 5-parametric family of  spline quasi-interpolants. In Section \ref{parameters} we present some strategies to fix the free parameters and in Section \ref{tests} we provide some numerical tests confirming the approximation properties.

\section{Notations and preliminaries}\label{Notations_preliminaries}

Given a triangulation $\Delta$ of the real plane, a polynomial $p_{d}$ of
degree less than or equal to $d$ can be represented on each triangle $T$
induced by $\Delta$ with vertices $v_{1}=\left(  v_{1,1},v_{1,2}\right)
$, $v_{2}=\left(  v_{2,1},v_{2,2}\right)  $ and $v_{3}=\left(  v_{3,1}%
,v_{3,2}\right)  $ in terms of its Bernstein basis. If $\tau:=\left(  \tau
_{1},\tau_{2},\tau_{3}\right)  $ are the barycentric coordinates with respect
to $T$, defined by the equalities%
\[
\left(  x,y\right)  =\tau_{1}\left(  v_{1,1},v_{1,2}\right)  +\tau
_{2}\left(  v_{2,1},v_{2,2}\right)  +\tau_{3}\left(  v_{3,1},v_{3,2}%
\right)  \quad\text{and}\quad\tau_{1}+\tau_{2}+\tau_{3}=1
\]
for $\left(  x,y\right)  \in T$, then%
\begin{equation}
p_{d}\left(  x,y\right)  =\sum_{\left\vert \alpha\right\vert =d}b_{\alpha}%
^{d}B_{\alpha}^{d}\left(  \tau\right)  , \label{p}%
\end{equation}
where $\left\vert \alpha\right\vert :=\alpha_{1}+\alpha_{2}+\alpha_{3}$ stands
for the length of the multi-index $\alpha:=\left(  \alpha_{1},\alpha
_{2},\alpha_{3}\right)  \in\mathbb{N}_{0}^{2}$, and%
\[
B_{\alpha}^{d}\left(  \tau\right)  :=\frac{d!}{\alpha!}\tau^{\alpha}=\frac
{d!}{\alpha_{1}!\alpha_{2}!\alpha_{2}!}\tau_{1}^{\alpha_{1}}\tau_{2}%
^{\alpha_{2}}\tau_{3}^{\alpha_{3}}%
\]
for the Bernstein polynomials of degree $d$ on $T$. The real numbers
$b_{\alpha}^{d}$ are said to be the Bernstein-B\'{e}zier (BB-) coefficients of
$p_{d}$ on $T$. They are related to the called domain points relative to $T$,
which are defined as $\xi_{\alpha}^{d}:=\frac{\alpha_{1}}{d}v_{1}+\frac
{\alpha_{2}}{d}v_{2}+\frac{\alpha_{3}}{d}v_{3}$. It is well-known that the
graph of the surface $z=p\left(  x,y\right)  $ on $T$ lies in the convex hull
of the set $\left\{  \left(  \xi_{\alpha}^{d},b_{\alpha}^{d}\right)
,\left\vert \alpha\right\vert =d\right\}  $ of control points.

We are interested in constructing spline functions on the triangulation $\Delta$,
so it is useful to recall the conditions on the BB-coefficients of its
restrictions to the triangles that guarantee the $C^{r}$ regularity.

Suppose a polynomial $\widetilde{p}_{d}$ of degree $d$ is defined on the
triangle $\widetilde{T}$ of vertices $v_{4}$, $v_{3}$ and $v_{2}$, thus
sharing with $T$ the edge defined by $v_{2}$ and $v_{3}$. Then,%
\begin{equation}
\widetilde{p}_{d}\left(  x,y\right)  =\sum_{\left\vert \alpha\right\vert
=d}\widetilde{b}_{\alpha}^{d}\widetilde{B}_{\alpha}^{d}\left(  \widetilde
{\tau}\right)  ,\ \left(  x,y\right)  \in\widetilde{T}, \label{ptilde}%
\end{equation}
where $\left\{  \widetilde{B}_{\alpha}^{d}\left(  \widetilde{\tau}\right)
,\left\vert \alpha\right\vert =d\right\}  $ is the basis of Bernstein
polynomials of $\mathbb{P}_{d}\left(  \widetilde{T}\right)  $, which are
expressed in terms of the corresponding barycentric coordinates $\widetilde
{\tau}:=\left(  \widetilde{\tau}_{1},\widetilde{\tau}_{2},\widetilde{\tau}%
_{3}\right)  $. Then, the following result holds \cite[Lemma 2.29]{ls}:

\begin{lemma}\label{C0C1}
The polynomials $p$ and $\widetilde{p}$ given in (\ref{p}) and
(\ref{ptilde}), respectively, join with $C^{r}$ smoothness across the edge
defined by $v_{2}$ and $v_{3}$ if%
\[
\widetilde{b}_{n,j,k}=\sum_{\left\vert \beta\right\vert =n}b_{\beta
_{1},k+\beta_{2},j+\beta_{3}}B_{\beta}^{n}\left(  \tau\right)
,\ j+k=d-n,\ n=0,\ldots,r,
\]
where $\tau$ denotes the barycentric coordinates of $v_{4}$ with respect to
$T$.
\end{lemma}

In particular, with $\tau=\left(  \tau_{1},\tau_{2}\text{,}\tau_{3}\right)  $,
for $C^{1}$ smoothness the equalities%
\begin{align}
\widetilde{b}_{0,j,k}  &  =b_{0,k,j},\ j+k=d,\label{C0}\\
\widetilde{b}_{0,j,k}  &  =\tau_{1}b_{1,k,j}+\tau_{2}b_{0,k+1,j}+\tau
_{3}b_{0,k,j+1},\ j+k=d-1, \label{C1}%
\end{align}
are required \cite[eq (2.50) in Thm. 2.28 ]{ls}.

In this work, we consider the uniform triangulation $\Delta_{3}$ defined
by the vectors $e_{1}:=\left(  h,h\right)  $, $e_{2}:=\left(  h,-h\right)  $
and $e_{3}:=e_{1}+e_{2}$, with a given $h>0$. It gives rise to vertices
$v_{i,j}:=ie_{1}+je_{2}$,$\ i,j\in\mathbb{Z}$. It also produces two types of
triangles. The first one is $T_{i,j}:=\left[  v_{i,j},v_{i+1,j+1}%
,v_{i+1,j}\right]  $ and the other one $\widetilde{T}_{i,j}:=\left[
v_{i,j},v_{i+1,j+1},v_{i,j+1}\right]  $ (see Fig. \ref{triangulacion}).

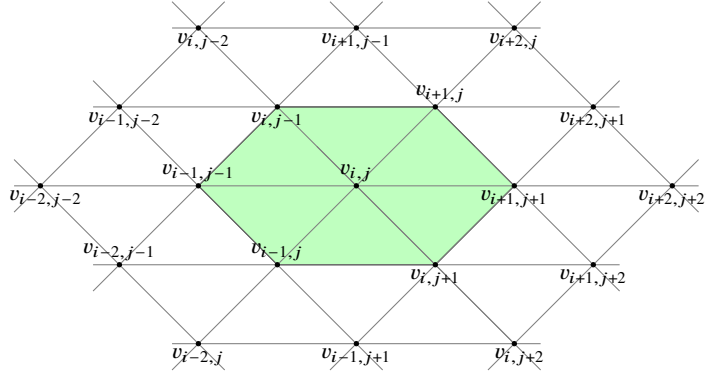
\begin{figure}[ptb]
\centering
\begin{tikzpicture}[scale=0.35]
\path[draw, fill=green!25, opacity=0.15] (7,10)--(10,7)--(16.,7)--(19.,10.)--(16.,13.)--(10,13)--(7,10)--cycle;
\draw[gray,thin] (6,4) --(20,4);   
\draw[gray,thin] (3,7) --(23,7);   
\draw[gray,thin] (0,10) --(26,10);   
\draw[gray,thin] (3,13) --(23,13);   
\draw[gray,thin] (6,16) --(20,16);   
\draw[gray,thin] (0,9) --(8,17);  
\draw[gray,thin] (3,6) --(14,17); 
\draw[gray,thin] (6,3) --(20,17) ;  
\draw[gray,thin] (12,3) --(23,14);
\draw[gray,thin] (18,3) --(26,11); 
\draw[gray,thin] (6,17) --(20,3);  
\draw[gray,thin] (0,11) --(8,3); 
\draw[gray,thin]  (3,14) --(14,3); 
\draw[gray,thin] (12,17) --(23,6); 
\draw[gray,thin] (18,17) --(26,9); 
\draw[gray,thin] (11.5,11.5) --(19,4);
\draw[gray,thin] (17.5,11.5) --(19,10); 
\draw[fill] (1,10) circle(0.08cm);
\node[black] at (1.2,9.5) {\small $v_{i-2,j-2}$};
\draw[fill] (4,13) circle(0.08cm);
\node[black] at (4.2,12.5) {\small $v_{i-1,j-2}$};
\draw[fill] (7,16) circle(0.08cm);
\node[black] at (7.2,15.5) {\small $v_{i,j-2}$};
\draw[fill] (4,7) circle(0.08cm);
\node[black] at (4.,7.5) {\small $v_{i-2,j-1}$};
\draw[fill] (7,10) circle(0.08cm);
\node[black] at (7.,10.5) {\small $v_{i-1,j-1}$};
\draw[fill] (10,13) circle(0.08cm);
\node[black] at (10.,12.5) {\small $v_{i,j-1}$};
\draw[fill] (13,16) circle(0.08cm);
\node[black] at (13.,15.5) {\small  $v_{i+1,j-1}$};
\draw[fill] (7,4) circle(0.08cm);
\node[black] at (7.,3.5) {\small $v_{i-2,j}$};
\draw[fill] (10,7) circle(0.08cm);
\node[black] at (10.,7.5) {\small $v_{i-1,j}$};
\draw[fill] (13,10) circle(0.08cm);
\node[black] at (13.,10.5) {\small  $v_{i,j}$};
\draw[fill] (16,13) circle(0.08cm);
\node[black] at (16.2,13.5) {\small $v_{i+1,j}$};
\draw[fill] (19,16) circle(0.08cm);
\node[black] at (19.,15.5) {\small $v_{i+2,j}$};
\draw[fill] (13,4) circle(0.08cm);
\node[black] at (13.,3.5) {\small $v_{i-1,j+1}$};
\draw[fill] (16,7) circle(0.08cm);
\node[black] at (16.,6.5) {\small $v_{i,j+1}$};
\draw[fill] (19,10) circle(0.08cm);
\node[black] at (19.,9.5) {\small $v_{i+1,j+1}$};
\draw[fill,] (22,13) circle(0.08cm);
\node[black] at (22.,12.5) {\small $v_{i+2,j+1}$};
\draw[fill] (19,4) circle(0.08cm);
\node[black] at (19.2,3.5) {\small $v_{i,j+2}$};
\draw[fill] (22,7) circle(0.08cm);
\node[black] at (22.,6.5) {\small $v_{i+1,j+2}$};
\draw[fill] (25,10) circle(0.08cm);
\node[black] at (25,9.5) {\small $v_{i+2,j+2}$};
\end{tikzpicture}
\caption{The triangulation $\Delta_{3}$ and the hexagon $H_{i,j}$ centered
at $v_{i,j}$ defined by its six closest vertices.}%
\label{triangulacion}%
\end{figure}

The quasi-interpolating splines will be constructed in the space
\[
S_3^1\left(  \Delta_{3}\right)  :=\left\{  s\in C^{1}\left(  R^{2}\right)
:s_{\mid T}\in P_{3}\text{ for all }T\in\Delta_{3}\right\}  .
\]
According to (\ref{p}), their BB-coefficients on each triangle of
$\Delta_{3}$ will be directly setting. Given $s\in S_3^1 \left(
\Delta_{3}\right)  $, its restriction to a specific triangle $T$
(equal to $T_{i,j}$ or $\widetilde{T}_{i,j}$) can be written as
\[
s_{\mid T}=%
{\displaystyle\sum\limits_{\left\vert \alpha\right\vert =3}}
b_{\alpha}^{T}B_{\alpha}^{T},
\]
where a superscript is used to show that the Bernstein polynomials depend on
the triangle considered and the BB-coefficients will be defined from the
available information on the function to be approximated.

In each triangle a cubic spline is uniquely determined by ten BB-coefficients,
linked to ten domain points. The subset $D$ consisting of the domain points of
all the triangles can be written as $D=\bigcup_{i,j}D_{i,j}$, where%
\[
D_{i,j}:=\left\{  v_{i,j},c_{i,j},\widetilde{c}_{i,j}\right\}  \cup\left\{
u_{i,j}^{k,m},k,m\in\left\{  -1,0,1\right\}  ,k+m\neq0\right\}  ,
\]
$c_{i,j}$ and $\widetilde{c}_{i,j}$ being the barycenters of $T_{i,j}$ and
$\widetilde{T}_{i,j}$, respectively, and%
\[
u_{i,j}^{k,m}:=\frac{1}{3}\left(  2v_{i,j}+v_{i+k,j+m}\right)  .
\]
This partition is essential for the construction to be proposed. As the
triangulation is uniform, it will be sufficient to define the BB-coefficients
associated with the points in $D_{i,j}$. Fig. \ref{domainpoints} shows the
domain points linked to the BB-coefficients that determine a cubic spline in
the triangles $T_{i,j}$ and $\widetilde{T}_{i,j}$.

\begin{figure}[ptb]
\centering
\begin{tikzpicture}[scale=0.85]
\foreach \x in {1,3,5,7}
\draw[fill,black] (\x,4) circle (0.05cm);
\foreach \x in {2,4,6}
\draw[fill,black] (\x,5) circle (0.05cm);
\foreach \x in {2,4,6}
\draw[fill,black] (\x,3) circle (0.05cm);
\foreach \x in {3,5}
\draw[fill,black] (\x,2) circle (0.05cm);
\foreach \x in {3,5}
\draw[fill,black] (\x,6) circle (0.05cm);
\draw[fill,black] (4,1) circle (0.05cm);
\draw[fill,black] (4,7) circle (0.05cm);
\draw[gray, thin] (1,4) -- (4,1) -- (7,4) -- (4,7) -- (1,4) --cycle;
\draw[gray, thin] (1,4) -- (7,4);
\node[black] at (7.1,4.3) {\small {$v_{i+1,j+1}$}};
\node[black] at (5,4.3) {\small {$ u_{i+1,j+1}^{-1,-1}$}};
\node[black] at (3,4.3) {\small {$u_{i,j}^{1,1}$}};
\node[black] at (1,4.3) {\small  {$v_{i,j}$}};
\node[black] at (2,5.3) {\small {$u_{i,j}^{1,0}$}};
\node[black] at (4,5.3) {\small {$ c_{i,j}$}};
\node[black] at (6,5.3) {\small {$u_{i+1,j+1}^{0,-1}$}};
\node[black] at (3,6.3)  {\small {$u_{i+1,j}^{-1,0}$}};
\node[black] at (5,6.3) {\small {$u_{i+1,j}^{0,1}$}};
\node[black] at (4.,7.2) {\small {$v_{i+1,j}$}};
\node[black] at (2,3.3) {\small {$u_{i,j}^{0,1}$}};
\node[black] at (4,3.3) {\small {$ \widetilde{c}_{i,j}$}};
\node[black] at (6,3.3)  {\small {$u_{i+1,j+1}^{-1,0}$}};
\node[black] at (3,2.3) {\small {$u_{i,j+1}^{0,1}$}};
\node[black] at (5,2.3)  {\small {$u_{i,j+1}^{1,0}$}};
\node[black] at (4,1.3) {\small {$v_{i,j+1}$}};
\node[red] at (1.35,5.8) {$T_{i,j}$};
\node[red] at (1.35,2.5) {$\widetilde{T}_{i,j}$};
\end{tikzpicture}
\caption{The domain points in $T_{i,j}$ and $\widetilde{T}_{i,j}$.}%
\label{domainpoints}%
\end{figure}
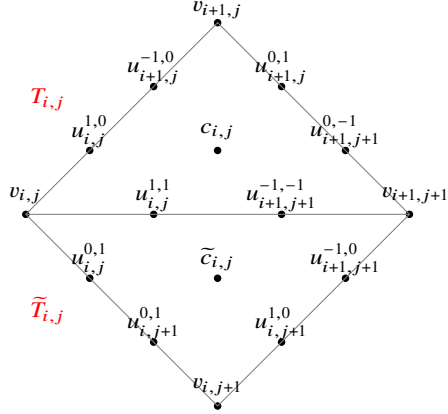

\section{$C^{1}$ cubic Hermite quasi-interpolation}\label{cubic}

In this section, we define quasi-interpolating splines $Qf\in S_3^1\left(  \Delta_{3}\right)  $ to a given function $f\in C^{1}\left(
\RR^{2}\right)  $ by assuming that the values of $f$ and its gradient
at the vertices are known. The BB-coefficients of $Qf$ are set on
each triangle $T$ as follows:
\begin{align}
Qf_{\mid T_{i,j}}  &  =V_{i,j}B_{3,0,0}^{T_{i,j}}+U_{i,j}^{1,1}B_{2,1,0}%
^{T_{i,j}}+U_{i,j}^{1,0}B_{2,0,1}^{T_{i,j}}+U_{i+1,j+1}^{-1,-1}B_{1,2,0}%
^{T_{i,j}}\label{Q_T}\\
&  +C_{i,j}B_{1,1,1}^{T_{i,j}}+U_{i+1,j}^{-1,0}B_{1,0,2}^{T_{i,j}}%
+V_{i+1,j+1}B_{0,3,0}^{T_{i,j}}+U_{i+1,j+1}^{0,-1}B_{0,2,1}^{T_{i,j}%
}\nonumber\\
&  +U_{i+1,j}^{0,1}B_{0,1,2}^{T_{i,j}}+V_{i+1,j}B_{0,0,3}^{T_{i,j}%
},\nonumber\\
Qf_{\mid\widetilde{T}_{i,j}}  &  =V_{i,j+1}B_{3,0,0}^{\widetilde{T}_{i,j}%
}+U_{i,j+1}^{1,0}B_{2,1,0}^{\widetilde{T}_{i,j}}+U_{i,j+1}^{0,1}%
B_{2,0,1}^{\widetilde{T}_{i,j}}+U_{i+1,j+1}^{-1,0}B_{1,2,0}^{\widetilde
{T}_{i,j}}\nonumber\\
&  +\widetilde{C}_{i,j}B_{1,1,1}^{\widetilde{T}_{i,j}}+U_{i,j}^{0,1}%
B_{1,0,2}^{\widetilde{T}_{i,j}}+V_{i+1,j+1}B_{0,3,0}^{\widetilde{T}_{i,j}%
}+U_{i+1,j+1}^{-1,-1}B_{0,2,1}^{\widetilde{T}_{i,j}}\nonumber\\
&  +U_{i,j}^{1,1}B_{0,1,2}^{\widetilde{T}_{i,j}}+V_{i,j}B_{0,0,3}%
^{\widetilde{T}_{i,j}}.\nonumber
\end{align}
Note that the BB-coefficients have been named as their corresponding domain
points using capital letters.

Expressions (\ref{Q_T}) involve the BB-coefficients associated with four
vertices, ten domain points of type $u$ and those of the two barycentres.
Taking into account that $\Delta_{3}$ is a uniform partition, it will be
sufficient to define the BB-coefficients of the domain points appearing in
$D_{i,j}$. They will be linear combinations of the values of $f$ and its first
order partial derivatives $\partial_{1,0}f$ and $\partial_{0,1}f$ at the seven
vertices in the hexagon $H_{i,j}$ (see Fig. \ref{triangulacion}). For example,
the BB-coefficient associated with the domain point $v_{i,j}$ has the
following form:

\begin{align}
V_{i,j} &  =\alpha_{0,0,0}\ f\left(  v_{i,j}\right)  +\alpha_{0,0,1}\ f\left(
v_{i+1,j+1}\right)  +\alpha_{0,0,2}\ f\left(  v_{i+1,j}\right)  +\alpha
_{0,0,3}\ f\left(  v_{i,j-1}\right)  \nonumber\\
&  +\alpha_{0,0,4}\ f\left(  v_{i-1,j-1}\right)  +\alpha_{0,0,5}\ f\left(
v_{i-1,j}\right)  +\alpha_{0,0,6}\ f\left(  v_{i,j+1}\right)  \nonumber\\
&  +\alpha_{1,0,0}\ \partial_{1,0}f\left(  v_{i,j}\right)  h+\alpha
_{1,0,1}\ \partial_{1,0}f\left(  v_{i+1,j+1}\right)  h+\alpha_{1,0,2}%
\ \partial_{1,0}f\left(  v_{i+1,j}\right)  h\nonumber\\
&  +\alpha_{1,0,3}\ \partial_{1,0}f\left(  v_{i,j-1}\right)  h+\alpha
_{1,0,4}\ \partial_{1,0}f\left(  v_{i-1,j-1}\right)  h+\alpha_{1,0,5}%
\ \partial_{1,0}f\left(  v_{i-1,j}\right)  h\label{Vij}\\
&  +\alpha_{1,0,6}\ \partial_{1,0}f\left(  v_{i,j+1}\right)  h+\alpha
_{0,1,0}\ \partial_{0,1}f\left(  v_{i,j}\right)  h+\alpha_{0,1,1}%
\ \partial_{0,1}f\left(  v_{i+1,j+1}\right)  h\nonumber\\
&  +\alpha_{0,1,2}\ \partial_{0,1}f\left(  v_{i+1,j}\right)  h+\alpha
_{0,1,3}\ \partial_{0,1}f\left(  v_{i,j-1}\right)  h+\alpha_{0,1,4}%
\ \partial_{0,1}f\left(  v_{i-1,j-1}\right)  h\nonumber\\
&  +\alpha_{0,1,5}\ \partial_{0,1}f\left(  v_{i-1,j}\right)  h+\alpha
_{0,1,6}\ \partial_{0,1}f\left(  v_{i,j+1}\right)  h.\nonumber
\end{align}

This expression can be simplified if three \textit{masks} $\alpha
_{0,0}:=\left(  \alpha_{0,0,\ell}\right)  _{0\leq\ell\leq 6}$, $\alpha
_{1,0}:=\left(  \alpha_{1,0,\ell}\right)  _{0\leq\ell\leq 6}$ and $\alpha
_{0,1}:=\left(  \alpha_{0,1,\ell}\right)  _{0\leq\ell\leq 6}$ are introduced,
as well the notation $g_{i,j}:=\left(  g\left(  v_{i,j}\right)  ,g\left(
v_{i+1,j+1}\right)  ,g\left(  v_{i+1,j}\right)  ,g\left(  v_{i,j-1}\right)
,g\left(  v_{i-1,j-1}\right)  ,g\left(  v_{i-1,j}\right)  ,g\left(
v_{i,j+1}\right)  \right) $ is introduced for a given function $g$. Thus,
equality (\ref{Vij}) can be written as%
\begin{equation}
V_{i,j}=\alpha_{0,0}\ f_{i,j}+\alpha_{1,0}\ {h} \ \partial_{1,0}f_{i,j}+\alpha
_{0,1}\ {h} \ \partial_{0,1}f_{i,j}. \label{Vijmod}%
\end{equation}
Similarly, for $k,m\in\left\{  -1,0,1\right\}  $ such that $k+m\neq0$, we
write%
\begin{equation}
U_{i,j}^{k,m}=\beta_{0,0}^{k,m}\ f_{i,j}+\beta_{1,0}^{k,m}\ {h} \ \partial_{1,0}%
f_{i,j}+\beta_{0,1}^{k,m}\ {h} \ \partial_{0,1}f_{i,j}. \label{Uij}%
\end{equation}
Finally, for the BB-coefficients $C_{i,j}$ and $\widetilde{C}_{i,j}$ relative
to the barycenters, we write%
\begin{equation}
C_{i,j}=\gamma_{0,0}\ f_{i,j}+\gamma_{1,0}\ {h} \ \partial_{1,0}f_{i,j}+\gamma
_{0,1}\ {h} \ \partial_{0,1}f_{i,j} \label{Bij}%
\end{equation}
and%
\begin{equation}
\widetilde{C}_{i,j}=\widetilde{\gamma}_{0,0}\ f_{i,j}+\widetilde{\gamma}%
_{1,0}\ {h} \ \partial_{1,0}f_{i,j}+\widetilde{\gamma}_{0,1}\ {h} \ \partial_{0,1}\ f_{i,j}.
\label{Bijtilde}%
\end{equation}
All masks $\alpha_{0,0}$, $\alpha_{1,0}$, $\alpha_{0,1}$, $\beta_{0,0}^{k,m}$,
$\beta_{1,0}^{k,m}$, $\beta_{0,1}^{k,m}$, $\gamma_{0,0}$, $\gamma_{1,0}$,
$\gamma_{0,1}$, $\widetilde{\gamma}_{0,0}$, $\widetilde{\gamma}_{1,0}$ and
$\widetilde{\gamma}_{0,1}$ must be computed in order to produce a $C^{1}$
cubic quasi-interpolant $Qf$. The previous local and linear construction results in the quasi-interpolation operator $\mathcal{Q}:C^{1}\left(  \RR^{2}\right)  \longrightarrow
S_3^1\left(  \Delta_{3}\right)  $ defined by $\mathcal{Q}[f]:=Qf$. It can only reproduce $\mathbb{P}_2$ since the order of approximation of $S_3^1\left(  \Delta_{3}\right)$  is only three \cite{deBoorJia}. This will be the exactness required of the operator.

As far as the regularity of $Qf$ is concerned, since the triangulation is
uniform, it is sufficient to impose it on the three edges emanating from the
vertex $v_{0,0}$. Therefore, conditions (\ref{C0}) and (\ref{C1}) must be
satisfied. The former are automatically satisfied by construction, so the
class $C^{1}$ must be imposed by requiring the fulfillment of (\ref{C1}). In
the case of the edge $\left[  v_{0,0},v_{1,1}\right]  $, the values $\tau_{1}
$, $\tau_{2} $ and $\tau_{3} $ to be used correspond to the barycentric
coordinates of the vertex $v_{0,1}$ with respect to the triangle $T_{0,0}$,
which are $\left(  1,1,-1\right)  $. For the edge $\left[  v_{0,0}%
,v_{1,0}\right]  $, the coordinates of $v_{0,-1}$ with respect to $T_{0,0}$
are also equal to $\left(  1,1,-1\right)  $. The same result holds for the
barycentric coordinates of $v_{-1,0}$ with respect to $T_{0,0}$, needed by the
$C^{1}$ continuity across the edge $\left[  v_{0,0},v_{-1,0}\right]  $.

\begin{proposition}$Qf$ is $C^{1}$ continuous if and only if%
\begin{equation}%
\begin{small}
\begin{array}
[c]{rlrlll}%
{\small U}_{0,0}^{0,-1}+{\small U}_{0,0}^{1,1}= & {\small V}_{0,0}%
+{\small U}_{0,0}^{1,1}, & {\small \widetilde{C} _{0,-1}}%
+{\small C_{0,0}}= & {\small U}_{0,0}^{1,0}+{\small U}_{1,0}^{-1,0}, &
{\small U}_{1,0}^{-1,1}+{\small U}_{1,0}^{0,1}= & {\small U}_{1,0}%
^{-1,0}+{\small V}_{1,0},\\
{\small V}_{0,0}+{\small U}_{0,0}^{1,1}= & {\small U}_{0,0}^{1,0}%
+{\small U}_{0,0}^{0,1}, & {\small C_{0,0}}%
+{\small \widetilde {C}_{0,0}}= & {\small U}_{0,0}^{1,1}+{\small U}%
_{1,1}^{-1,-1}, & {\small U}_{1,1}^{-1,-1}+{\small V}_{1,1}= & {\small U}%
_{1,1}^{0,-1}+{\small U}_{1,1}^{-1,0},\\
{\small U}_{0,0}^{1,1}+{\small U}_{0,0}^{-1,0}= & {\small V}_{0,0}%
+{\small U}_{0,0}^{0,1}, & \ {\small C_{-1,0}}%
+{\small \widetilde{C}_{0,0}}= & {\small U}_{0,0}^{0,1}+{\small U}%
_{0,1}^{0,1}, & {\small U}_{0,1}^{0,1}+{\small V}_{0,1}= & {\small U}%
_{0,1}^{-1,-1}+{\small U}_{0,1}^{1,0}.
\end{array}
\end{small}
\label{eqsC1}%
\end{equation}
\end{proposition}

\begin{proof}
Equations (\ref{eqsC1}) are the result of applying the equalities (\ref{C0}) and (\ref{C1}) taking into account the domain points involved (see Fig. \ref{dibujoC1}).
\end{proof}

\begin{figure}[ptb]
\centering
\begin{tikzpicture}[scale=0.7]
\path[draw, fill=green!15, opacity=0.35] (2,7)--(4,7)--(5.,6)--(3.,6.)--cycle;
\path[draw, fill=green!15, opacity=0.35] (1,6)--(3,6)--(4,5)--(2,5.)--cycle;
\path[draw, fill=green!15, opacity=0.35] (0,5)--(2,5)--(3,4)--(1,4.)--cycle;
\path[draw, fill=yellow!15, opacity=0.35] (2,1)--(4,1)--(5.,2)--(3.,2.)--cycle;
\path[draw, fill=yellow!15, opacity=0.35] (1,2)--(3,2)--(4,3)--(2,3.)--cycle;
\path[draw, fill=yellow!15, opacity=0.35] (0,3)--(2,3)--(3,4)--(1,4.)--cycle;
\foreach \x in {1,3,5,7}
\draw[fill,black] (\x,4) circle (0.05cm);
\foreach \x in {2,4,6}
\draw[fill,black] (\x,5) circle (0.05cm);
\foreach \x in {2,4,6}
\draw[fill,black] (\x,3) circle (0.05cm);
\foreach \x in {3,5}
\draw[fill,black] (\x,2) circle (0.05cm);
\foreach \x in {3,5}
\draw[fill,black] (\x,6) circle (0.05cm);
\draw[fill,black] (4,1) circle (0.05cm);
\draw[fill,black] (4,7) circle (0.05cm);
\draw[gray, thin] (1,4) -- (4,1) -- (7,4) -- (4,7) -- (1,4) --cycle;
\draw[gray, thin] (1,4) -- (7,4);
\draw[gray, ultra thin] (0,5) -- (4,5);
\draw[gray, ultra thin] (0,3) -- (4,3);
\draw[gray, ultra thin] (1,2) -- (5,2);
\draw[gray, ultra thin] (1,6) -- (5,6);
\draw[gray, ultra thin] (2,1) -- (4,1);
\draw[gray, ultra thin] (2,1) -- (6,5);
\draw[gray, ultra thin] (1,2) -- (5,6);
\draw[gray, ultra thin] (0,3) -- (4,7);
\draw[gray, ultra thin] (2,7) -- (6,3);
\draw[gray, ultra thin] (1,6) -- (5,2);
\draw[gray, ultra thin] (0,5) -- (4,1);
\draw[gray, ultra thin] (1,2) -- (5,2);
\draw[gray, ultra thin] (2,7) -- (4,7);
\draw[gray, ultra thin] (1,6) -- (3,6);
\draw[gray, ultra thin] (2,3) -- (2,3);
\draw[gray, ultra thin] (1,2) -- (3,2);
\draw[gray, ultra thin] (2,1) -- (4,1);
\draw[fill,black] (2,1) circle (0.04cm);
\node[black] at (1.6,1) {\small {$u_{i,j+1}^{-1,-1}$}};%
\draw[fill,black] (1,2) circle (0.04cm);
\node[black] at (0.6,2) {\small {$c_{i+1,j}$}};%
\draw[fill,black] (0,3) circle (0.04cm);
\node[black] at (0.6,3) {\small {$u_{i,j}^{-1,0}$}};%
\draw[fill,black] (0,5) circle (0.04cm);
\node[black] at (0.6,5) {\small {$u_{i,j}^{0,-1}$}};%
\draw[fill,black] (1,6) circle (0.04cm);
\node[black] at (0.6,5.85) {\small {$\widetilde{c}_{i,j-1}$}};%
\draw[fill,black] (2,7) circle (0.04cm);
\node[black] at (1.6,7) {\small {$u_{i+1,j}^{-1,-1}$}};%
\node[black] at (7.1,4.3) {\small {$v_{i+1,j+1}$}};
\node[black] at (5,4.3) {\small {$ u_{i+1,j+1}^{-1,-1}$}};
\node[black] at (3,4.3) {\small {$u_{i,j}^{1,1}$}};
\node[black] at (1,4.3) {\small {$v_{i,j}$}};
\node[black] at (2,5.3) {\small {$u_{i,j}^{1,0}$}};
\node[black] at (4,5.3) {\small {$ c_{i,j}$}};
\node[black] at (6,5.3) {\small {$u_{i+1,j+1}^{0,-1}$}};
\node[black] at (3,6.3)  {\small {$u_{i+1,j}^{-1,0}$}};
\node[black] at (5,6.3) {\small {$u_{i+1,j}^{0,1}$}};
\node[black] at (4.,7.2) {\small {$v_{i+1,j}$}};
\node[black] at (2,3.3) {\small {$u_{i,j}^{0,1}$}};
\node[black] at (4,3.3) {\small {$ \widetilde{c}_{i,j}$}};
\node[black] at (6,3.3)  {\small {$u_{i+1,j+1}^{-1,0}$}};
\node[black] at (3,2.3) {\small {$u_{i,j+1}^{0,1}$}};
\node[black] at (5,2.3)  {\small {$u_{i,j+1}^{1,0}$}};
\node[black] at (4,1.3) {\small {$v_{i,j+1}$}};
\end{tikzpicture}
\hspace{0.25cm}
\begin{tikzpicture}[scale=0.7]
\path[draw, fill=gray!15, opacity=0.35] (1,4)--(2,5)--(3.,4.)--(2.,3)--cycle;
\path[draw, fill=gray!15, opacity=0.35] (3,4)--(4,5)--(5.,4.)--(4.,3)--cycle;
\path[draw, fill=gray!15, opacity=0.35] (5,4)--(6,5)--(7.,4.)--(6.,3)--cycle;
\foreach \x in {1,3,5,7}
\draw[fill,black] (\x,4) circle (0.05cm);
\foreach \x in {2,4,6}
\draw[fill,black] (\x,5) circle (0.05cm);
\foreach \x in {2,4,6}
\draw[fill,black] (\x,3) circle (0.05cm);
\foreach \x in {3,5}
\draw[fill,black] (\x,2) circle (0.05cm);
\foreach \x in {3,5}
\draw[fill,black] (\x,6) circle (0.05cm);
\draw[fill,black] (4,1) circle (0.05cm);
\draw[fill,black] (4,7) circle (0.05cm);
\draw[gray, thin] (1,4) -- (4,1) -- (7,4) -- (4,7) -- (1,4) --cycle;
\draw[gray, thin] (1,4) -- (7,4);
\draw[gray, ultra thin] (0,5) -- (4,5);
\draw[gray, ultra thin] (0,3) -- (4,3);
\draw[gray, ultra thin] (1,2) -- (5,2);
\draw[gray, ultra thin] (1,6) -- (5,6);
\draw[gray, ultra thin] (2,1) -- (4,1);
\draw[gray, ultra thin] (2,1) -- (6,5);
\draw[gray, ultra thin] (1,2) -- (5,6);
\draw[gray, ultra thin] (0,3) -- (4,7);
\draw[gray, ultra thin] (2,7) -- (6,3);
\draw[gray, ultra thin] (1,6) -- (5,2);
\draw[gray, ultra thin] (0,5) -- (4,1);
\draw[gray, ultra thin] (1,2) -- (5,2);
\draw[gray, ultra thin] (2,7) -- (4,7);
\draw[gray, ultra thin] (1,6) -- (3,6);
\draw[gray, ultra thin] (2,3) -- (2,3);
\draw[gray, ultra thin] (1,2) -- (3,2);
\draw[gray, ultra thin] (2,1) -- (4,1);
\draw[fill,black] (2,1) circle (0.04cm);
\node[black] at (1.6,1) {\small {$u_{i,j+1}^{-1,-1}$}};%
\draw[fill,black] (1,2) circle (0.04cm);
\node[black] at (0.6,2) {\small {$c_{i+1,j}$}};%
\draw[fill,black] (0,3) circle (0.04cm);
\node[black] at (0.6,3) {\small {$u_{i,j}^{-1,0}$}};%
\draw[fill,black] (0,5) circle (0.04cm);
\node[black] at (0.6,5) {\small {$u_{i,j}^{0,-1}$}};%
\draw[fill,black] (1,6) circle (0.04cm);
\node[black] at (0.6,6) {\small {$\widetilde{c}_{i,j-1}$}};%
\draw[fill,black] (2,7) circle (0.04cm);
\node[black] at (1.6,7) {\small {$u_{i+1,j}^{-1,-1}$}};%
\node[black] at (7.1,4.3) {\small {$v_{i+1,j+1}$}};
\node[black] at (5,4.3) {\small {$ u_{i+1,j+1}^{-1,-1}$}};
\node[black] at (3,4.3) {\small {$u_{i,j}^{1,1}$}};
\node[black] at (1,4.3) {\small  {$v_{i,j}$}};
\node[black] at (2,5.3) {\small {$u_{i,j}^{1,0}$}};
\node[black] at (4,5.3) {\small {$ c_{i,j}$}};
\node[black] at (6,5.3) {\small {$u_{i+1,j+1}^{0,-1}$}};
\node[black] at (3,6.3)  {\small {$u_{i+1,j}^{-1,0}$}};
\node[black] at (5,6.3) {\small {$u_{i+1,j}^{0,1}$}};
\node[black] at (4.,7.2) {\small {$v_{i+1,j}$}};
\node[black] at (2,3.3) {\small {$u_{i,j}^{0,1}$}};
\node[black] at (4,3.3) {\small {$ \widetilde{c}_{i,j}$}};
\node[black] at (6,3.3)  {\small {$u_{i+1,j+1}^{-1,0}$}};
\node[black] at (3,2.3) {\small {$u_{i,j+1}^{0,1}$}};
\node[black] at (5,2.3)  {\small {$u_{i,j+1}^{1,0}$}};
\node[black] at (4,1.3) {\small {$v_{i,j+1}$}};
\end{tikzpicture}
\caption{The conditions equivalent to the $C^{1}$ smoothness of $Qf$.}%
\label{dibujoC1}%
\end{figure}
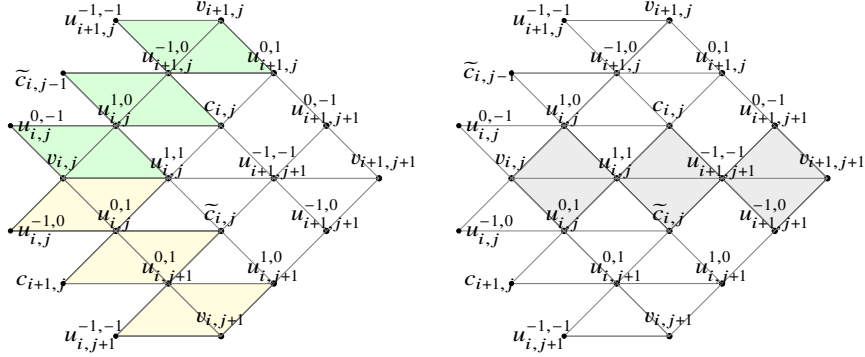

Each of the above nine equalities corresponds to a linear functional whose
action on an arbitrary function $f$ must be zero. The image by such a linear
functional is a linear combination of values of $f$, $\partial_{1,0}f$ and
$\partial_{0,1}f$ at the vertices of the set $S:=D_{0,0}\cup D_{1,1}\cup
D_{1,0}\cup D_{0,-1}\cup D_{-1,-1}\cup D_{-1,0}\cup D_{0,1}$. The coefficients
of such a linear combination must be zero, giving rise to linear equations
that must be satisfied.

\begin{proposition}
The problem of finding a quasi-interpolation operator
$\mathcal{Q}$ exact on $\pp_{2}$ and such that $Qf \in S_3^1\left( \Delta_3\right) $ is defined by means of BB-coefficients given in
(\ref{Vijmod})-(\ref{Bijtilde}) has a 5-parametric family of solutions.
\end{proposition}

\begin{proof}
The solution of the problem is found by solving the system of equations provided by the $C^{1}$ smoothness and those resulting from imposing equality between each of the BB-coefficients of $Qm_{\mu}$ and the corresponding one of $m_{\mu}$, $\left\vert \mu\right\vert \leq2$, at each of the triangles $T_{0,0}$ and $\widetilde{T}_{0,0}$, where $m_{\mu}\left(  x,y\right)  :=x^{\mu_{1}}y^{\mu_{2}}$. This solution is calculated by means of a Computer Algebra System. The parameters are $\alpha_{0,0,2}$, $\alpha_{1,0,2}$, $\alpha_{0,1,2}$, $\alpha_{0,0,3}$ and $\alpha_{1,0,3}$.
\end{proof}

\section{Choice of parameters}\label{parameters}

The existence of degrees of freedom makes it possible to construct
quasi-interpolants with additional properties. The
exactness on $\pp_{2}$ of the operator $\mathcal{Q}^{\ast}$ provided by
the above proposition implies that for each triangle $T$ induced by the triangulation $\Delta_3$ the quasi-interpolation error $\left\Vert f-\mathcal{Q}^{\ast}\left[  f\right]  \right\Vert _{C^{1},T}$ relative to $T$ is of order $\mathcal{O}\left(h^{3}\right)$, {where $\left\Vert f\right\Vert _{C^{1},T}:=\left\Vert f\right\Vert _{\infty
.T}+h\left\Vert \partial_{1,0}f\right\Vert _{\infty,T}+h\left\Vert
\partial_{0,1}f\right\Vert _{\infty,T}$} (see,
e.g. \cite{DeVoreLorentz1993}). {More precisely,%

\[
\left\Vert f-\mathcal{Q}^{\ast}\left[  f\right]  \right\Vert _{C^{1},T}%
\leq\left(  1+\left\Vert \mathcal{Q}^{\ast}\right\Vert _{C^{1}}\right)
\operatorname*{dist}\nolimits_{C^{1},T}\left(  f,\mathbb{P}_{2}\right),
\]
with $\operatorname*{dist}\nolimits_{C^{1},T}\left(  f,\mathbb{P}_{2}\right)
=\inf_{p\in\mathbb{P}_{2}}\left\Vert f-p\right\Vert _{C^{1},T}$.
}
It is straightforward to prove that
\begin{align*}
\left\Vert \mathcal{Q}^{\ast}\right\Vert _{C^{1}}  & \leq\max\left\{
\left\Vert \alpha_{0,0}\right\Vert _{1}+\left\Vert \alpha_{1,0}\right\Vert
_{1}+\left\Vert \alpha_{0,1}\right\Vert _{1},\right.  \\
& \left.  \left\Vert \gamma_{0,0}\right\Vert _{1}+\left\Vert \gamma
_{1,0}\right\Vert _{1}+\left\Vert \gamma_{0,1}\right\Vert _{1},\left\Vert
\widetilde{\gamma}_{0,0}\right\Vert _{1}+\left\Vert \widetilde{\gamma}%
_{1,0}\right\Vert _{1}+\left\Vert \widetilde{\gamma}_{0,1}\right\Vert
_{1};\right.  \\
& \left.  \left\Vert \beta_{0,0}^{k,m}\right\Vert _{1}+\left\Vert \beta
_{1,0}^{k,m}\right\Vert _{1}+\left\Vert \beta_{0,1}^{k,m}\right\Vert
_{1},\ k,m\in\left\{  -1,0,1\right\}  ,\ k+m\neq0\right\},
\end{align*}
where $\left\Vert v\right\Vert _{1}:=\sum_{\ell=1}^{N}\left\vert v_{\ell
}\right\vert $ for $v\in\mathbb{R}^{N}$.

It is not possible to achieve a higher order of global convergence than above, but it is feasible to attain higher orders of convergence at specific points, obtaining quasi-interpolants that are often called super-convergent. Specifically, we will ask
that the quasi-interpolation error 
$$\varepsilon\left[  f\right]  \left(
q\right)  :=f\left(  q\right)  -\mathcal{Q}^{\ast}\left[  f\right]  \left(
q\right)  
$$ be of order greater than or equal to four at the midpoints of the edges of the triangulation.

\begin{proposition}\label{superconvergencia}
It is satisfied that the quasi-interpolation error $\varepsilon\left( f\right) $ is of order four at the midpoints of the sides of $\Delta_{3}$ if and only if%
\begin{align*}
\alpha_{0,0,2}  & =\lambda,\alpha_{0,0,3}=\frac{1}{6}\left(  -5+12\lambda
\right)  ,\alpha_{1,0,2}=\frac{{1}}{36}\left(  1-18\lambda\right)
,\alpha_{1,0,3}=\frac{{1}}{12}\left(  5-18\lambda\right)  ,\\
\alpha_{0,1,2}  & =-\frac{h}{9},
\end{align*}
$\lambda$ being an arbitrary value.
\end{proposition}

\begin{proof}
The midpoints of the triangulation edges are the midpoints $e_{i,j}^{k,\ell}$ of the edges $\left[  v_{i,j},v_{i+k,j+\ell}\right]
$, $k,\ell\in\left\{  0,1\right\}  $, $k+\ell\neq0$, $i,j\in \mathbb{Z}$.
Since the triangulation is uniform, it is sufficient to prove the claim for the midpoints of the triangle $T_{0,0}$. The exactness of $\mathcal{Q}^{\ast}$ on $\mathbb{P}_{2}$ implies that the linear functional $\varepsilon$ is null on this space.
The value of $\mathcal{Q}^{\ast}\left[  f\right]  $ on the triangle $T_{0,0}$ is determined from the BB-coefficients%
\[
\left\{  V_{0,0},U_{0,0}^{1,1},U_{0,0}^{1,0},U_{1,1}^{-1,-1},T_{0,0}%
,U_{1,0}^{-1,0},V_{1,1},U_{1,1}^{0,-1},U_{1,0}^{0,1},V_{1,0}\right\}  ,
\]
which are defined from the masks and the values of $f$,
$\partial_{1,0}f$ and $\partial_{0,1}f$ at the vertices in $S$.
The de Casteljau algorithm allows to easily calculate the value of the
quasi-interpolant of each of the $m_{\mu}$ cubic monomials, $\left\vert
\mu\right\vert =3$, at the midpoints $e_{0,0}^{1,1}$, $e_{0,0}^{0,1}$ and
$e_{1,0}^{0,1}$ lying in the edges $\left[  v_{0,0},v_{1,1}\right]  $, $\left[
v_{0,0},v_{1,0}\right]  $ and $\left[  v_{1,0},v_{1,1}\right]  $, whose barycentric coordinates are $\left(  \frac{1}{2},\frac{1}{2},0\right)  $,
$\left(  \frac{1}{2},0,\frac{1}{2}\right)  $ and $\left(  0,\frac{1}{2},\frac
{1}{2}\right)  $, respectively.
Regarding the midpoint $e_{0,0}^{1,1}$, the following results hold:%
\begin{align*}
\varepsilon\left[  m_{3,0}\right]  \left(  e_{0,0}^{1,1}\right)    &
=\frac{h^{3}}{4}\left(  12\alpha_{0,0,2}-6\alpha_{0,0,3}-5\right)  ,\\
\varepsilon\left[  m_{2,1}\right]  \left(  e_{0,0}^{1,1}\right)    &
=-\frac{h^{2}}{4}\left(  4h\alpha_{0,0,2}+24\alpha_{1,0,2}-2h\alpha
_{0,0,3}-8\alpha_{1,0,3}+h\right)  ,\\
\varepsilon\left[  m_{1,2}\right]  \left(  e_{0,0}^{1,1}\right)    &
=\frac{h^{2}}{108}\left(  540h\alpha_{0,0,2}+288\alpha_{1,0,2}+864\alpha
_{0,1,2}+18h\alpha_{0,0,3}\right.  \\
& \left.  +288\alpha_{1,0,3}-17h\right)  ,\\
\varepsilon\left[  m_{0,3}\right]  \left(  e_{0,0}^{1,1}\right)    &
=\frac{h^{2}}{36}\left(  108h\alpha_{0,0,2}+72\alpha_{1,0,2}+18h\alpha
_{0,0,3}+72\alpha_{1,0,3}-17h\right)  .
\end{align*}
One will have superconvergence at $e_{0,0}^{1,1}$ if and only if the above expressions are equal to zero, which is equivalent to the claim in the statement.
It is straightforward to check that these same conditions guarantee superconvergence at $e_{0,0}^{0,1}$ and $e_{1,0}^{0,1}$.
\end{proof}

Figures \ref{mascaras1parametrovertices} and
\ref{mascaras1parametrobaricentros} show the masks provided by Proposition
\ref{superconvergencia} {for vertices and barycenters, respectively.}

\begin{figure}[ptb]
\centering
\begin{tikzpicture}[scale=0.75]
\node[red] at (0.2,1.6)  {$\alpha_{0,0}$};
\draw[gray,thin]  (0,1)--(1,0) --(3,0) --(4,1) --(3,2)--(1,2)--cycle; 
\draw[gray, ultra thin,dashed]  (0,1)--(4,1);
\draw[gray, ultra thin,dashed]  (1,0)--(3,2);
\draw[gray, ultra thin,dashed]   (1,2)--(3,0);
\draw[fill,black] (1,0) circle(0.04cm);
\node[black] at (1,-0.2)  {\footnotesize {$-\frac{1}{6}\left(  7+12\lambda\right)  $} };
\draw[fill,black] (3,0) circle(0.04cm);
\node[black] at (3,-0.2)  {\footnotesize {$1-\lambda$} };
\draw[fill,black] (0,1) circle(0.04cm);
\node[black] at (-0.25,1)  {\footnotesize {$-\frac{2}{3}$} };
\draw[fill,black] (2,1) circle(0.04cm);
\node[black] at (2.,1.25)  {\footnotesize {$\frac{1}{3}$} };
\draw[fill,black] (4,1) circle(0.04cm);
\node[black] at (4.,1.2)  {\footnotesize {$0$} };
\draw[fill,black] (1,2) circle(0.04cm);
\node[black] at (1.,2.2)  {\footnotesize {$\frac{1}{6}\left(  12\lambda-5\right)  $} };
\draw[fill,black] (3,2) circle(0.04cm);
\node[black] at (3.,2.2)  {\footnotesize {$\lambda$}};
\end{tikzpicture}
\begin{tikzpicture}[scale=0.75]
\node[red] at (0.2,1.6)  {$\alpha_{1,0}$};
\draw[gray,thin]  (0,1)--(1,0) --(3,0) --(4,1) --(3,2)--(1,2)--cycle; 
\draw[gray, ultra thin,dashed]  (0,1)--(4,1);
\draw[gray, ultra thin,dashed]  (1,0)--(3,2);
\draw[gray, ultra thin,dashed]   (1,2)--(3,0);
\draw[fill,black] (1,0) circle(0.04cm);
\node[black] at (0.8,-0.2)  {\footnotesize {$\frac{1}{12}(18\lambda-13)$} };
\draw[fill,black] (3,0) circle(0.04cm);
\node[black] at (3.2,-0.2)  {\footnotesize {$\frac{1}{36}(18\lambda-17)$} };
\draw[fill,black] (0,1) circle(0.04cm);
\node[black] at (-0.25,1)  {\footnotesize {$-\frac{2}{9}$} };
\draw[fill,black] (2,1) circle(0.04cm);
\node[black] at (2.,1.25)  {\footnotesize {$-\frac{2}{3}$} };
\draw[fill,black] (4,1) circle(0.04cm);
\node[black] at (4.,1.2)  {\footnotesize {$0$} };
\draw[fill,black] (1,2) circle(0.04cm);
\node[black] at (0.8,2.2)  {\footnotesize {$\frac{1}{12}(5-18\lambda)$} };
\draw[fill,black] (3,2) circle(0.04cm);
\node[black] at (3.2,2.2)  {\footnotesize {$\frac{1}{36}(1-18\lambda)$}};
\end{tikzpicture}
\begin{tikzpicture}[scale=0.75]
\node[red] at (0.2,1.6)  {$\alpha_{0,1}$};
\draw[gray,thin]  (0,1)--(1,0) --(3,0) --(4,1) --(3,2)--(1,2)--cycle; 
\draw[gray, ultra thin,dashed]  (0,1)--(4,1);
\draw[gray, ultra thin,dashed]  (1,0)--(3,2);
\draw[gray, ultra thin,dashed]   (1,2)--(3,0);
\draw[fill,black] (1,0) circle(0.04cm);
\node[black] at (1,-0.2)  {\footnotesize {$\frac{1}{18}(13-18\lambda)$} };
\draw[fill,black] (3,0) circle(0.04cm);
\node[black] at (3,-0.2)  {\footnotesize {$\frac{1}{9}$} };
\draw[fill,black] (0,1) circle(0.04cm);
\node[black] at (0,1.25)  {\footnotesize {$1-2\lambda$} };
\draw[fill,black] (2,1) circle(0.04cm);
\node[black] at (2.,1.25) {\footnotesize {$1-2\lambda$} };
\draw[fill,black] (4,1) circle(0.04cm);
\node[black] at (4.,1.2)  {\footnotesize {$0$} };
\draw[fill,black] (1,2) circle(0.04cm);
\node[black] at (1.3,2.2)  {\footnotesize {$\frac{1}{18}(5-18\lambda)$} };
\draw[fill,black] (3,2) circle(0.04cm);
\node[black] at (3.2,2.2)  {\footnotesize {$-\frac{1}{9}$}};
\end{tikzpicture}
\caption{Masks associated with vertices for achieving $C^{1}$ smoothness and
superconvergence at midpoints of edges (and also at vertices), and ensuring
exactness on $\pp_{2}$.}%
\label{mascaras1parametrovertices}%
\end{figure}
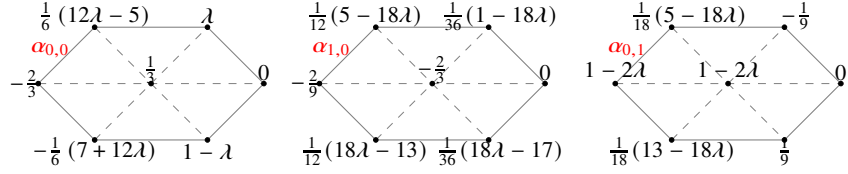

\begin{figure}[ptb]
\centering
\begin{tikzpicture}[scale=0.75]
\node[red] at (0.2,1.6)  {$\gamma_{0,0}$};
\draw[gray,thin]  (0,1)--(1,0) --(3,0) --(4,1) --(3,2)--(1,2)--cycle; 
\draw[gray, ultra thin,dashed]  (0,1)--(4,1);
\draw[gray, ultra thin,dashed]  (1,0)--(3,2);
\draw[gray, ultra thin,dashed]   (1,2)--(3,0);
\draw[fill,black] (1,0) circle(0.04cm);
\node[black] at (1,-0.2)  {\footnotesize {$0 $} };
\draw[fill,black] (3,0) circle(0.04cm);
\node[black] at (3,-0.2)  {\footnotesize {$1-\lambda$} };
\draw[fill,black] (0,1) circle(0.04cm);
\node[black] at (-0.25,1)  {\footnotesize {$0$} };
\draw[fill,black] (2,1) circle(0.04cm);
\node[black] at (2.,1.25)  {\footnotesize {$\frac{1}{3}\left(  2-9\lambda\right)  $} };
\draw[fill,black] (4,1) circle(0.04cm);
\node[black] at (4.,1.2)  {\footnotesize {$1-\lambda$} };
\draw[fill,black] (1,2) circle(0.04cm);
\node[black] at (1.,2.2)  {\footnotesize {$\frac{1}{6}\left(  6\lambda-5\right)  $} };
\draw[fill,black] (3,2) circle(0.04cm);
\node[black] at (3.,2.2)  {\footnotesize {$\frac{1}{6}\left(  24\lambda-5\right)  $}};
\end{tikzpicture}
\begin{tikzpicture}[scale=0.75]
\node[red] at (0.2,1.6)  {$\gamma_{1,0}$};
\draw[gray,thin]  (0,1)--(1,0) --(3,0) --(4,1) --(3,2)--(1,2)--cycle; 
\draw[gray, ultra thin,dashed]  (0,1)--(4,1);
\draw[gray, ultra thin,dashed]  (1,0)--(3,2);
\draw[gray, ultra thin,dashed]   (1,2)--(3,0);
\draw[fill,black] (1,0) circle(0.04cm);
\node[black] at (1,-0.2)  {\footnotesize {$0 $} };
\draw[fill,black] (3,0) circle(0.04cm);
\node[black] at (3,-0.2)  {\footnotesize {$\frac{1}{36}\left(  18\lambda-17\right)  $} };
\draw[fill,black] (0,1) circle(0.04cm);
\node[black] at (-0.25,1)  {\footnotesize {$0$} };
\draw[fill,black] (2,1) circle(0.04cm);
\node[black] at (1.5,1.25)  {\footnotesize {$\frac{1}{12}\left( 30\lambda-23\right)  $} };
\draw[fill,black] (4,1) circle(0.04cm);
\node[black] at (3.9,1.2)  {\footnotesize {$\frac{1}{36}\left(  18\lambda-17\right)  $} };
\draw[fill,black] (1,2) circle(0.04cm);
\node[black] at (0.8,2.2)  {\footnotesize {$\frac{1}{18}\left(  7-18\lambda\right)  $} };
\draw[fill,black] (3,2) circle(0.04cm);
\node[black] at (3.2,2.2)  {\footnotesize {$\frac{1}{36}\left(  17-90\lambda\right)  $}};
\end{tikzpicture}
\begin{tikzpicture}[scale=0.75]
\node[red] at (0.2,1.6)  {$\gamma_{0,1}$};
\draw[gray,thin]  (0,1)--(1,0) --(3,0) --(4,1) --(3,2)--(1,2)--cycle; 
\draw[gray, ultra thin,dashed]  (0,1)--(4,1);
\draw[gray, ultra thin,dashed]  (1,0)--(3,2);
\draw[gray, ultra thin,dashed]   (1,2)--(3,0);
\draw[fill,black] (1,0) circle(0.04cm);
\node[black] at (1.,-0.2)  {\footnotesize {$0 $} };
\draw[fill,black] (3,0) circle(0.04cm);
\node[black] at (3,-0.3)  {\footnotesize {$\frac{1}{9} $} };
\draw[fill,black] (0,1) circle(0.04cm);
\node[black] at (0.3,1)  {\footnotesize {$0$} }; 
\draw[fill,black] (2,1) circle(0.04cm);
\node[black] at (2.,1.25)  {\footnotesize {$\frac{1}{3}\left( 7-12\lambda\right)  $} };
\draw[fill,black] (4,1) circle(0.04cm);
\node[black] at (4.,1.3)  {\footnotesize {$\frac{1}{9}$} };
\draw[fill,black] (1,2) circle(0.04cm);
\node[black] at (0.9,2.2)  {\footnotesize {$\frac{1}{18}\left(  7-18\lambda\right)  $} };
\draw[fill,black] (3,2) circle(0.04cm);
\node[black] at (3.1,2.2)  {\footnotesize {$\frac{1}{18}\left(  1-18\lambda\right)  $}};
\end{tikzpicture}
\begin{tikzpicture}[scale=0.75]
\node[red] at (0.2,1.6)  {$\widetilde{\gamma}_{0,0}$};
\draw[gray,thin]  (0,1)--(1,0) --(3,0) --(4,1) --(3,2)--(1,2)--cycle; 
\draw[gray, ultra thin,dashed]  (0,1)--(4,1);
\draw[gray, ultra thin,dashed]  (1,0)--(3,2);
\draw[gray, ultra thin,dashed]   (1,2)--(3,0);
\draw[fill,black] (1,0) circle(0.04cm);
\node[black] at (0.9,-0.3)  {\footnotesize {$\frac{1}{6}\left(  1-6\lambda\right)  $} };
\draw[fill,black] (3,0) circle(0.04cm);
\node[black] at (3,-0.3)  {\footnotesize {$\frac{1}{6}\left(  19-24\lambda\right)  $} };
\draw[fill,black] (0,1) circle(0.04cm);
\node[black] at (-0.25,1)  {\footnotesize {$0$} };
\draw[fill,black] (2,1) circle(0.04cm);
\node[black] at (2.,1.25)  {\footnotesize {$\frac{1}{3}\left( 9\lambda-7\right)  $} };
\draw[fill,black] (4,1) circle(0.04cm);
\node[black] at (4.,1.2)  {\footnotesize {$\lambda$} };
\draw[fill,black] (1,2) circle(0.04cm);
\node[black] at (1.,2.2)  {\footnotesize {$0 $} };
\draw[fill,black] (3,2) circle(0.04cm);
\node[black] at (3.,2.2)  {\footnotesize {$\lambda $}};
\end{tikzpicture}
\begin{tikzpicture}[scale=0.75]
\node[red] at (0.2,1.6)  {$\widetilde{\gamma}_{1,0}$};
\draw[gray,thin]  (0,1)--(1,0) --(3,0) --(4,1) --(3,2)--(1,2)--cycle; 
\draw[gray, ultra thin,dashed]  (0,1)--(4,1);
\draw[gray, ultra thin,dashed]  (1,0)--(3,2);
\draw[gray, ultra thin,dashed]   (1,2)--(3,0);
\draw[fill,black] (1,0) circle(0.04cm);
\node[black] at (0.9,-0.3)  {\footnotesize {$\frac{1}{18}\left(  18\lambda-11\right)  $} };
\draw[fill,black] (3,0) circle(0.04cm);
\node[black] at (3.2,-0.3)  {\footnotesize {$\frac{1}{36}\left( 90\lambda-73\right)  $} };
\draw[fill,black] (0,1) circle(0.04cm);
\node[black] at (-0.25,1)  {\footnotesize {$0$} };
\draw[fill,black] (2,1) circle(0.04cm);
\node[black] at (1.8,1.25)  {\footnotesize {$\frac{1}{12}\left( 7-30\lambda\right)  $} };
\draw[fill,black] (4,1) circle(0.04cm);
\node[black] at (4.,1.2)  {\footnotesize {$\frac{1}{36}\left( 1-18\lambda\right)  $} };
\draw[fill,black] (1,2) circle(0.04cm);
\node[black] at (1.,2.2)  {\footnotesize {$0 $} };
\draw[fill,black] (3,2) circle(0.04cm);
\node[black] at (3.,2.2)  {\footnotesize {$\frac{1}{36}\left( 1-18\lambda\right)  $}};
\end{tikzpicture}
\begin{tikzpicture}[scale=0.75]
\node[red] at (0.2,1.6)  {$\widetilde{\gamma}_{0,1}$};
\draw[gray,thin]  (0,1)--(1,0) --(3,0) --(4,1) --(3,2)--(1,2)--cycle; 
\draw[gray, ultra thin,dashed]  (0,1)--(4,1);
\draw[gray, ultra thin,dashed]  (1,0)--(3,2);
\draw[gray, ultra thin,dashed]   (1,2)--(3,0);
\draw[fill,black] (1,0) circle(0.04cm);
\node[black] at (0.8,-0.2)  {\footnotesize {$\frac{1}{18}\left(  11-18\lambda\right)  $} };
\draw[fill,black] (3,0) circle(0.04cm);
\node[black] at (3,-0.2)  {\footnotesize {$\frac{1}{18}\left( 17-18\lambda\right)  $} };
\draw[fill,black] (0,1) circle(0.04cm);
\node[black] at (-0.25,1)  {\footnotesize {$0$} };
\draw[fill,black] (2,1) circle(0.04cm);
\node[black] at (2.,1.25)  {\footnotesize {$\frac{1}{3}\left( 5-12\lambda\right)  $} };
\draw[fill,black] (4,1) circle(0.04cm);
\node[black] at (4.,1.2)  {\footnotesize {$-\frac{1}{9}  $} };
\draw[fill,black] (1,2) circle(0.04cm);
\node[black] at (1.,2.2)  {\footnotesize {$0 $} };
\draw[fill,black] (3,2) circle(0.04cm);
\node[black] at (3.,2.2)  {\footnotesize {$-\frac{1}{9}  $}};
\end{tikzpicture}
\caption{Masks associated with barycenters.}%
\label{mascaras1parametrobaricentros}%
\end{figure}
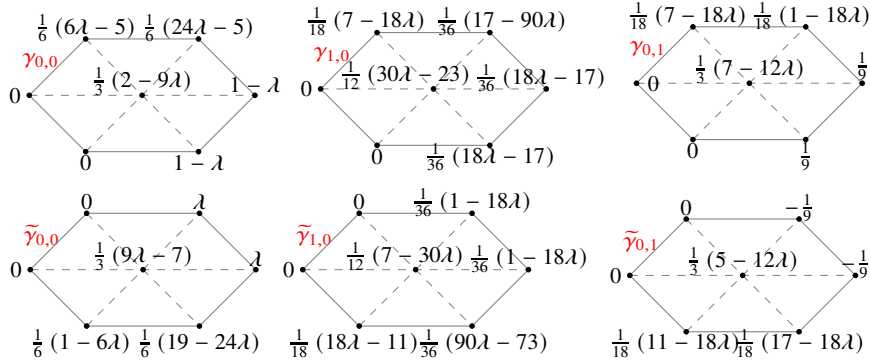

Next, the maks for $U_{i,j}^{k,m}$. Firstly, those associated with the values
of $f$ at the vertices:%
\begin{align*}
\beta_{0,0}^{1,1}  &  =\left(  -\frac{1}{3},0,2\lambda,\frac{1}{6}\left(
6\lambda-5\right)  ,0,\frac{1}{6}\left(  1-6\lambda\right)  ,2\left(
1-\lambda\right)  \right)  ,\\
\beta_{0,0}^{1,0}  &  =\left(  \frac{3}{2}\left(  1-2\lambda\right)
,0,2\lambda,\frac{1}{3}\left(  9\lambda-5\right)  ,\frac{1}{6}\left(
1-6\lambda\right)  ,0,1-\lambda\right) \\
\beta_{0,0}^{0,-1}  &  =\left(  \frac{1}{6}\left(  13-18\lambda\right)
,0,\lambda,\frac{1}{3}\left(  12\lambda-5\right)  ,-\frac{1}{2}\left(
2\lambda+1\right)  ,1-\lambda,0\right)  ,\\
\beta_{0,0}^{-1,-1}  &  =\left(  1,0,0,\frac{1}{6}\left(  18\lambda-5\right),
-\frac{4}{3},\frac{1}{6}\left(  13-18\lambda\right)  ,0\right)  ,\\
\beta_{0,0}^{-1,0}  &  =\left(  \frac{1}{6}\left(  18\lambda-5\right)
,0,0,\lambda,\frac{1}{2}\left(  2\lambda-3\right)  ,\frac{1}{3}\left(
7-12\lambda\right)  ,1-\lambda\right) \\
\beta_{0,0}^{0,1}  &  =\left(  \frac{3}{2}\left(  2\lambda-1\right)
,0,\lambda,0,\frac{1}{6}\left(  6\lambda-5\right)  ,\frac{1}{3}\left(
4-9\lambda\right)  ,2\left(  1-\lambda\right)  \right)  .
\end{align*}
Regarding the masks linked to $\partial_{1,0}f$, the following masks were
obtained:%
\begin{align*}
\beta_{1,0}^{1,1}  &  =\left(  -\frac{8}{9},0,\frac{1-18\lambda}{18},\frac{7-18\lambda}{18},0,\frac{18\lambda-11}{18},\frac{18\lambda-17}{18}\right)
,\\
\beta_{1,0}^{1,0}  &  =\left(  \frac{2\left(  9\lambda-8\right) }{9}%
,0,\frac{1-18\lambda}{18},\frac{29-90\lambda}{36},\frac{18\lambda-1}{18},0,\frac{18\lambda-17}{36}\right)  ,\\
\beta_{1,0}^{0,-1}  &  =\left(  \frac{2\left(  9\lambda-7\right)}%
{9},0,\frac{1-18\lambda}{36},\frac{5-18\lambda}{6},\frac{6\lambda-5}{6},\frac{18\lambda-17}{36},0\right)  ,\\
\beta_{1,0}^{-1,-1}  &  =\left(  -\frac{4}{9},0,0,\frac{2\left(
2-9\lambda\right)}{9},-\frac{4}{9},\frac{2\left(  9\lambda-7\right) }%
{9},0\right)  ,\\
\beta_{1,0}^{-1,0}  &  =\left(  \frac{2\left(  2-9\lambda\right) }%
{9},0,0,\frac{1-18\lambda}{36},\frac{1-6\lambda}{6},\frac{18\lambda-13}{6},\frac{18\lambda-17}{36}\right)  ,\\
\beta_{1,0}^{0,1}  &  =\left(  \frac{2\left(  1-9\lambda\right)}{9}%
,0,\frac{1-18\lambda}{36},0,\frac{7-18\lambda}{18},\frac{90\lambda-61}{36},\frac{18\lambda-17}{18}\right)  .
\end{align*}

Finally, the masks associated with $\partial_{0,1}f$ are%
\begin{align*}
\beta_{0,1}^{1,1}  &  =\left(  2\left(  1-2\lambda\right)  ,0,-\frac{2}%
{9},\frac{7-18\lambda}{18},0,\frac{11-18\lambda}{18},\frac{2}{9}\right)  ,\\
\beta_{0,1}^{1,0}  &  =\left(  \frac{11-18\lambda}{6},0,-\frac{2}{9},\frac{2\left(  1-3\lambda\right) }{3},\frac{11-18\lambda}{18},0,\frac{1}{9}\right)  ,\\
\beta_{0,1}^{0,-1}  &  =\left(  \frac{5-6\lambda}{6},0,-\frac{1}{9},\frac{5-18\lambda}{9},\frac{29-54\lambda}{18},\frac{1}{9},0\right)  ,\\
\beta_{0,1}^{-1,-1}  &  =\left(  0,0,0,\frac{1-6\lambda}%
{6},2\left(  1-2\lambda\right)  ,\frac{5-6\lambda}{6},0\right)  ,\\
\beta_{0,1}^{-1,0}  &  =\left(  \frac{1-6\lambda}{6},0,0,-\frac{1}{9},\frac{25-54\lambda}{18},\frac{13-18\lambda}{9},\frac{1}{9}\right)  ,\\
\beta_{0,1}^{0,1}  &  =\left(  \frac{7-18\lambda}{6},0,-\frac{1}{9},0,\frac{7-18\lambda}{18},\frac{2\left(2-2\lambda\right)}{3},\frac{2}{9}\right)  .
\end{align*}

\begin{corollary}
The masks in Proposition \ref{superconvergencia} produce quasi-interpolants that are also superconvergent at the vertices.
\end{corollary}

\begin{proof}
It is sufficient to check that $\varepsilon\left[  m_{\mu}\right]  \left(  v_{0,0}\right)=0$, for $\left\vert \mu\right\vert =3$.
\end{proof}

After selecting four parameters by imposing superconvergence on the midpoints
of the sides, only one parameter remains, which is susceptible to be selected.
One possibility is to check the behaviour of the quasi-interpolation error at
the vertices of the corresponding quasi-interpolant $\mathcal{Q}_{\lambda
}^{\ast}\left[  f\right]  $. It is easy to prove that%
\[
\varepsilon\left[  m_{4,0}\right]  \left(  v_{0,0}\right)  =\varepsilon\left[
m_{0,4}\right]  \left(  v_{0,0}\right)  =-\frac{4}{3}h^{4},\varepsilon\left[
m_{2,2}\right]  \left(  v_{0,0}\right)  =\frac{4}{9}h^{4},\varepsilon\left[
m_{1,3}\right]  \left(  v_{0,0}\right)  =0
\]
and%
\[
\varepsilon\left[  m_{3,1}\right]  \left(  v_{0,0}\right)  =2h^{4}\left(
2\lambda-1\right)  .
\]
Therefore, the choice $\lambda=1/2$ produces quasi-interpolants $\mathcal{Q}%
_{1/2}^{\ast}\left[  f\right]  $ with symmetric behaviour with respect to the
errors at the vertices for the quartic monomials.

\section{Numerical tests}\label{tests}

We test the performance of the operator $\mathcal{Q}_{1/2}^{\ast}$ by
considering the classical Franke and Nielson test functions. They are
\cite{Franke,Nielson}%
\begin{align*}
f\left(  x,y\right)   &  =\frac{1}{2}\exp\left(  -\left(  (9x-7)^{2}+\frac
{1}{4}(9y-3)^{2}\right)  \right)  +\frac{3}{4}\exp\left(  -\frac{1}%
{49}(9x+1)^{2}-\frac{1}{10}(9y+1)\right) \\
&  -\frac{1}{5}\exp\left(  -(9x-4)^{2}-(9y-7)^{2}\right)  +\frac{3}{4}%
\exp\left(  -\left(  (9x-2)^{2}+(9y-2)^{2}\right)  \right)  ,\\
g\left(  x,y\right)   &  =\frac{1}{2}y\cos^{4}\left(  4\left(  x^{2}%
+y-1\right)  \right)  .
\end{align*}
Their quasi-interpolants are computed on $\Omega:=\left[  0,1\right]
\times\left[  0,1\right]  $ for which evaluations of $f$ and $g$ at points
outside $\Omega$ but close to its boundary are necessary.

Table~\ref{tests1} shows approximate values of the
$\left\Vert f-\mathcal{Q}_{1/2}^{\ast}\left[  f\right]  \right\Vert
_{\infty,\Omega}$ and $\left\Vert g-\mathcal{Q}_{1/2}^{\ast}\left[  g\right]
\right\Vert _{\infty,\Omega}$ for $h=1/n$, with $n=8,16,32,46,128$. They are
estimated from the values of the test function and its quasi-interpolant at 28 points in each triangle. It also includes the numerical approximation orders, computed as the rate
\[
\text{NCO}:=\log\left(  \frac{\mathtt{E}\left(  h_{2}\right)  }{\mathtt{E}%
\left(  h_{1}\right)  }\right)  /\log\left(  \frac{h_{2}}{h_{1}}\right)  ,
\]
where $\mathtt{E}\left(  h\right)  $ stands for the estimated error associated with the step-length $h$.

The results confirm the theoretical results regarding the performance of
$\mathcal{Q}_{1/2}^{\ast}$.

\begin{table}[ptb]
\centering%
\begin{tabular}
[c]{l|c|c|c|c|}\cline{2-5}
& \multicolumn{2}{|c|}{$f$} & \multicolumn{2}{|c|}{$g$}\\\hline
\multicolumn{1}{|c|}{$n$} & error & NCO & error & NCO\\\hline
\multicolumn{1}{|c|}{$8$} & \multicolumn{1}{|l|}{$3.624\times10^{-1}$} & $-$ &
\multicolumn{1}{|l|}{$5.258\times10^{-1}$} & $-$\\
\multicolumn{1}{|c|}{$16$} & \multicolumn{1}{|l|}{$8.836\times10^{-2}$} &
\multicolumn{1}{|l|}{$2.036$} & \multicolumn{1}{|l|}{$1.062\times10^{-1}$} &
\multicolumn{1}{|l|}{$2.307$}\\
\multicolumn{1}{|c|}{$32$} & \multicolumn{1}{|l|}{$8.742\times10^{-3}$} &
\multicolumn{1}{|l|}{$3.337$} & \multicolumn{1}{|l|}{$9.658\times10^{-3}$} &
\multicolumn{1}{|l|}{$3.459$}\\
\multicolumn{1}{|c|}{$64$} & \multicolumn{1}{|l|}{$7.303\times10^{-4}$} &
\multicolumn{1}{|l|}{$3.581$} & \multicolumn{1}{|l|}{$7.426\times10^{-4}$} &
\multicolumn{1}{|l|}{$3.701$}\\
\multicolumn{1}{|c|}{$128$} & \multicolumn{1}{|l|}{$7.550\times10^{-5}$} &
\multicolumn{1}{|l|}{$3.274$} & \multicolumn{1}{|l|}{$6.381\times10^{-5}$} &
\multicolumn{1}{|l|}{$3.541$}\\\hline
\end{tabular}
\caption{Estimations of the quasi-interpolation errors for Franke and Nielson
functions provided by the operator $\mathcal{Q}_{1/2}^{\ast}$ for $h=1/n$.}%
\label{tests1}%
\end{table}

\section{Conclusions}
In this paper, we have proposed the construction of $C^{1}$ cubic quasi-interpolants on a three-direction mesh of $\RR^{2}$. The quasi-interpolating splines have been defined by directly setting their BB-coefficients from point and gradient values in order to reproduce quadratic polynomials, the highest possible degree. The resulting spline depends on five parameters that we have fixed imposing additional properties. Finally, we have provided some numerical tests confirming the approximation properties.

\section*{Acknowledgements}
The authors would like to thank the referees for their comments, suggestions and proposed changes, which have greatly improved the original version.

The first and third authors are members of the research group FQM 191 \textit{Matem\'{a}tica Aplicada} funded by the PAIDI programme of the Junta de Andaluc\'{\i}a, Spain.

The second author is a member of the research group GNCS of Italy and acknowledges the support of the MUR Excellence Department Project awarded to the Department of Mathematics, University of Rome Tor Vergata, CUP E83C23000330006.

The fourth author is a member of the INdAM Research group GNCS of Italy and was supported by this group.


\begin{thebibliography}{99}                                                      
\bibitem {bj}M. Buhmann, J. J\"{a}ger. Quasi-Interpolation, Cambridge
University Press, 2022. 

\bibitem {ls}M.J. Lai, L.L. Schumaker. Spline functions on triangulations,
Cambridge University Press, 2007.  

\bibitem {SZ3}T. Sorokina, F. Zeilfelder. An explicit quasi-interpolation
scheme based on $C^{1}$ quartic splines on type-1 triangulations, Computer
Aided Geometric Design 25 (2008) 1--13. 

\bibitem{SorokinaZeilfelder2005} T. Sorokina, F. Zeilfelder, Optimal quasi-interpolation by quadratic $C^1$ splines on four-directional meshes. In: Chui, C., et al. (Eds.),
Approximation Theory, vol. XI. Gatlinburg 2004. Nashboro Press, Brentwood, TN, pp. 423–438. 

\bibitem {bcir2}D. Barrera, C. Dagnino, M.J. Ib\'a\~{n}ez, S. Remogna. Some
results on cubic and quartic quasi-interpolation of optimal approximation
order on type-1 triangulations, Rend. Semin. Mat. Univ. Politec. Torino 76(2)
(2018) 29--38. 

\bibitem {bcir3}D. Barrera, C. Dagnino, M.J. Ib\'a\~{n}ez, S. Remogna. Point
and differential $C^{1}$ quasi-interpolation on three direction meshes, J.
Comput. Appl. Math. 354 (2019) 373--389. 

\bibitem {bcir4}D. Barrera, C. Dagnino, M.J. Ib\'a\~{n}ez, S. Remogna.
Quasi-interpolation by $C^{1}$ quartic splines on type-1 triangulations, J.
Comput. Appl. Math. 349 (2019) 225--238. 

\bibitem {beir}D. Barrera, S. Eddargani, M.J. Ib\'{a}\~{n}ez, S. Remogna.
Spline quasi-interpolation in the Bernstein basis on the Powell-Sabin 6-split
of a type-1 triangulation, J. Comput. Appl. Math. 424 (2023) 115011. 

\bibitem{Proc2019} D. Barrera, C. Conti, C. Dagnino, M.J. Ib\'{a}\~{n}ez, S. Remogna, $C^1$-Quartic Butterfly-spline interpolation on type-1 triangulations, in: G.E. Fasshauer, M. Neamtu, L.L. Schumaker (Eds.), Approximation Theory XVI, Nashville, TN, USA, May 19--22, 2019. In: Springer Proceedings in Mathematics $\&$ Statistics, Vol. 336, 2021, pp. 11--26. 

\bibitem{MMAS} F. J. Ariza‐L\'{o}pez, D. Barrera, S. Eddargani, M.J. Ib\'{a}\~{n}ez, J. F. Reinoso, Spline quasi‐interpolation in the Bernstein basis and its application to digital elevation models, Mathematical Methods in the Applied Sciences, 46 (2023) 1687--1698. 

\bibitem{BEIR23} D. Barrera, S. Eddargani, M.J. Ib\'{a}\~{n}ez, S. Remogna, Low-degree spline quasi-interpolants in the Bernstein basis, Applied Mathematics and Computation  457 (2023) 128150.

\bibitem{deBoorJia} C. de Boor, Q. Jia, A sharp upper bound on the approximation order of
smooth bivariate pp functions, J. Approx. Theory 72 (1993) 24--33.


\bibitem {DeVoreLorentz1993}R.A. DeVore, G.G. Lorentz. Constructive
Approximation. Springer-Verlag, 1993. 

\bibitem {Franke}R. Franke, Scattered data interpolation: Tests of some
methods, Math. Comp. 38 (1982) 181--200. 

\bibitem {Nielson}G.M. Nielson, A first order blending method for triangles
based upon cubic interpolation, Internat. J. Numer. Methods Engrg. 15 (1978) 308--318. 
\end{thebibliography}
\end{document}